\pdfoutput=1
\documentclass[a4paper,10pt]{article}

\usepackage[a4paper, portrait, margin=2cm]{geometry}
\usepackage[utf8]{inputenc}
\usepackage{amsthm}
\usepackage{amsmath}
\usepackage{amssymb}
\usepackage[backend=biber,citestyle=numeric]{biblatex}
\usepackage{hyperref}
\usepackage{cleveref}
\usepackage{graphicx}
\usepackage{color}
\usepackage[affil-it]{authblk}
\usepackage{lipsum}
\usepackage{tikz}
\usepackage{pgfplots}
\usepackage{subcaption}
\usepackage{LaTeX_changes}
 
\newcommand{\SSd}{\mathbb{S}^{n-1}}
\newcommand{\C}{\mathbb{C}}
\newcommand{\R}{\mathbb{R}}
\newcommand{\AAn}[1]{\mathbb{A}^{#1,n}}
\newcommand{\df}{\textup{d}}
\newcommand{\dx}{\,\textup{d}}
\newcommand{\SYM}{{\mathcal{S}}}
\newcommand{\SII}{{\mathbb{S}^2}}
\newcommand{\AII}{{A}}
\newcommand{\AIV}{{\mathcal{A}}}
\newcommand{\diag}{{\textup{diag}}}
\newcommand{\tr}{{\textup{tr}}}

\newcommand{\gnuplotdefs}{
  Wm(x) = log(x/log(x/log(x/log(x/log(x/log(x/log(x/log(x/log(x/log(x/log(x/log(x/log(x/log(x/log(x/log(x/log(x/log(x/log(x/log(x/log(-x)))))))))))))))))))));
  Aiiii(a,b) = 0.5*(2*a - 2*((1 - a)/(2*b**0.5) - a*b)/(1/b**0.5 - b));
  Aiiii_asymptote(a) = (3*a - 1)/(2 + 1/a*Wm(-exp(2)*a/2));
  Aiiii_asymptote_2(a) = (3*a - 1)/(2 + 1/a*log(exp(2)*a/2));
  Aiiii_asymptote_3(a) = (3*a - 1)/(2 - 1/a);
  Aiiii_asymptote_4(a) = (3*a - 1)/(2 - 8/pi**2*(a - 1)**2);
  Aiiii_asymptote_5(a) = (3*a - 1)/(2 - 1.0/32.0*(pi - (32*(a-1) + pi**2)**0.5)**2);
  set samples 1000;
}

\newtheorem{lemma}{Lemma}
\newtheorem{remark}{Remark}

\title{Short note on a relation between the inverse of the cosine and Carlson's elliptic integral $R_D$}
\author[1]{Felix Ospald \thanks{Send any comments to (appreciated): \href{mailto:felix.ospald@gmail.com}{felix.ospald@gmail.com}}}
\author[1]{Roland Herzog}

\affil[1]{TU Chemnitz, Faculty of Mathematics, 09107 Chemnitz, Germany}

\makeatletter
\hypersetup{
  pdftitle={\@title},
  pdfsubject={We prove a simple relation for a special case of Carlson's elliptic integral $R_D$.
The findings are applied to derive explicit formulae for the asymptotics of certain moments of the angular central Gaussian distribution in terms of the second moment.},
  pdfauthor={\@author},
  pdfkeywords={injection molding, short fiber-reinforced plastics, closure approximation, exact closure, angular central Gaussian, elliptic integrals, Folgar-Tucker equation, Jeffery's equation}
}
\makeatother

\begin{document}

\maketitle

\begin{abstract}
We prove a simple relation for a special case of Carlson's elliptic integral $R_D$.
The findings are applied to derive explicit formulae for the asymptotics of certain moments of the angular central Gaussian distribution in terms of the second moment.
\end{abstract}

\section{Carlson's symmetric form}

\subsection{Primary definitions}

Using Carlson's symmetric integrals instead of Legendre's elliptic integrals unifies and simplifies the evaluation of elliptic integrals
to the evaluation to some basic elliptic integrals \cite{ZillCarlson1969,Carlson1988}, for which efficient algorithms \cite{Carlson1995} are available.
For a non-negative integer $n$, parameters $a, b_j \in \R$ and arguments $z_j \in \mathbb{C} \setminus (-\infty, 0]$ the symmetric integrals are defined in terms of the {multivariate hypergeometric function}
\begin{equation*}
  R_{-a}\left({b};{z}\right) := R_{-a}\left(b_{1},\dots,b_{n};z_{1}, \dots,z_{n}\right) := \frac{1}{\mathrm{B}\left(a,a^{\prime}\right)}\int_{0}^{\infty} \!\!\! t^{a^{\prime}-1}\prod^{n}_{j=1}(t+z_{j})^{-b_{j}}\dx{t}
\end{equation*}
with
\begin{equation*}
  a^{\prime} := -a+\sum_{j=1}^{n}b_{j} \quad \text{and} \quad 
  \mathrm{B}\left(a,b\right) 
    := \frac{\Gamma\left(a\right)\Gamma\left(b\right)}{\Gamma\left(a+b\right)}.
\end{equation*}
The $R$-function is homogeneous and of degree $-a$ in $z$ and normalized, i.e.
\begin{equation*}
  R_{-a}\left({b};\lambda{z}\right) = \lambda^{-a} R_{-a}\left({b};{z}\right) \quad \text{and} \quad
  R_{-a}\left({b};{1}\right) = 1.
\end{equation*}
Further it is symmetric in the variables $z_i$ and $z_j$ if the parameters $b_i$ and $b_j$ are equal.
Some frequently used elliptic integrals have a special name, in particular we have
\begin{align*}
  R_F(x,y,z) &:= R_{-\frac{1}{2}} \left(\frac{1}{2},\frac{1}{2},\frac{1}{2}; x, y, z\right) = \frac{1}{2} \int_{0}^{\infty} \!\!\!\! \frac{1}{s(t)} \, \dx{t}, \\
  R_J(x,y,z,p) &:= R_{-\frac{3}{2}} \left(\frac{1}{2},\frac{1}{2},\frac{1}{2}, 1; x, y, z, p\right) = \frac{3}{2} \int_{0}^{\infty} \!\!\!\! \frac{1}{s(t) (p + t)} \, \dx{t}, \\
  R_D(x,y,z) &:= R_{-\frac{3}{2}} \left(\frac{1}{2},\frac{1}{2},\frac{3}{2}; x, y, z\right) = R_J(x,y,z,z), \\ 
  R_C(x,y) &:= R_{-\frac{1}{2}} \left(\frac{1}{2}, 1; x, y\right) =  R_F(x,y,y),
\end{align*}
where $s(t) := \sqrt{x+t}\sqrt{y+t}\sqrt{z+t}$.

\subsection{Important properties}

We have the following relation between $R_F$ and $R_D$:
\begin{equation*}
  R_D(x,y,z) = -6 \frac{\partial R_F}{\partial z}(x,y,z)
\end{equation*}
and the partial derivatives of $R_D$ w.r.t.\ the symmetric variables are given by
\begin{equation*}
  \begin{aligned}
  \frac{\partial R_D}{\partial x}(x,y,z)
    &= \frac{R_D(y,z,x) - R_D(x,y,z)}{2(x-z)}, \\
  \frac{\partial R_D}{\partial y}(x,y,z)
    &= \frac{R_D(x,z,y) - R_D(x,y,z)}{2(y-z)}
  \end{aligned}
\end{equation*}
and for the derivative w.r.t.\ the last variable we have
\begin{equation*}
  \frac{\partial R_D}{\partial z}(x,y,z)
    = -\frac{3}{2}x^{-\tfrac{1}{2}}y^{-\tfrac{1}{2}}z^{-\tfrac{3}{2}} - \frac{\partial R_D}{\partial x}(x,y,z) - \frac{\partial R_D}{\partial y}(x,y,z).
\end{equation*}
Further there are the following symmetric connections between $R_D$, their arguments and $R_F$
\begin{align*}
  R_D(y,z,x) + R_D(x,z,y) + R_D(x,y,z) &= 3 (xyz)^{-\frac{1}{2}} \quad \text{and} \\
  x R_D(y,z,x) + y R_D(x,z,y) + z R_D(x,y,z) &= 3 R_F(x,y,z).
\end{align*}
For more properties cf.\ the work of B.~C.~Carlson and his chapter in \cite{OlverLozierBoisvertClark2010}.

\section{The relation}

\begin{lemma}
  For $x \in (0,1)$ we have
  \begin{equation}
    \label{eq:R_D_arccos}
    \frac{1}{3} x R_D(1, 1, x^2) = \frac{1}{2} \frac{\df^2}{\df{x}^2} \left[\arccos(x)\right]^2 .
  \end{equation}
\end{lemma}

\begin{proof}
  By \cite[eq.~19.20.20]{OlverLozierBoisvertClark2010} we have 
  \begin{equation*}
    \frac{1}{3} x R_D(1, 1, x^2) = \frac{x}{x^2-1}\left(R_{C}\left(x^2,1\right)-\frac{1}{|x|}\right)
  \end{equation*}
  and
  \begin{equation*}
     R_{C}\left(x^2,1\right) = \begin{cases}
       \frac{1}{\sqrt{1-x^2}}\operatorname{arccos}(x) & 0 \le x < 1, \\
       \frac{1}{\sqrt{x^2-1}}\operatorname{arctanh}(\sqrt{1-\frac{1}{x^2}}) & x > 1, \\
       1 & x = 1.
     \end{cases} 
  \end{equation*}
  For $x \in (0,1)$ this results in 
  \begin{equation*}
    \frac{1}{3} x R_D(1, 1, x^2) = \frac{1}{1-x^2}-\frac{x \arccos(x)}{\left(1-x^2\right)^{3/2}}.
  \end{equation*}
  On the other hand we have 
  \begin{equation*}
    \frac{1}{2} \frac{\df^2}{\df{x}^2} \left[\arccos(x)\right]^2 = \left(\frac{\df}{\df{x}}  \arccos(x) \right)^2 + \arccos(x) \frac{\df^2}{\df{x}^2} \arccos(x),
  \end{equation*}
  where 
  \begin{equation*}
     \frac{\df}{\df{x}} \arccos(x) = - \frac{1}{\sqrt{1-x^2}} \quad \text{and} \quad
     \frac{\df^2}{\df{x}^2} \arccos(x) = - \frac{x}{\left(1-x^2\right)^{3/2}}.
  \end{equation*}
  This also results in
  \begin{equation*}
    \frac{1}{2} \frac{\df^2}{\df{x}^2} \left[\arccos(x)\right]^2 =  \frac{1}{1-x^2}  - \frac{x \arccos(x)}{\left(1-x^2\right)^{3/2}} .
  \end{equation*}
\end{proof}

\begin{lemma}
  Similarly to \cref{eq:R_D_arccos} we have for $x \in (1,\infty)$
  \begin{equation}
    \label{eq:R_D_arctanh}
    \frac{1}{3} x R_D(1, 1, x^2) = -\frac{1}{2} \frac{\df^2}{\df{x}^2} \left[\operatorname{arctanh}\left(\sqrt{1-\frac{1}{x^2}}\right)\right]^2 .
  \end{equation}
\end{lemma}

\begin{remark}
   By analytic continuation, \cref{eq:R_D_arccos} 
   and \cref{eq:R_D_arctanh} 
   extend to the right half plane of complex numbers $\{ z \in \C: \Re(z) > 0 \}$.
   On this set both equations are identical. Due to the simplicity we prefer to use \cref{eq:R_D_arccos} in the following.
\end{remark}

\section{Limits and asymptotics of the inverse}

\begin{remark}[Limits]
  We have the following limits on the real axis
  \begin{equation*}
    \lim_{x \rightarrow 0^+} x R_D(1, 1, x^2) = 3
  \end{equation*}
  and in the complex plane we have
  \begin{equation*}
    \lim_{z \rightarrow 1} z R_D(1, 1, z^2) = 1 \quad \text{and} \quad
    \lim_{z \rightarrow \infty} z R_D(1, 1, z^2) = 0 .
  \end{equation*}
\end{remark}

\begin{remark}[Asymptotics of the inverse for $x \rightarrow \infty$]

  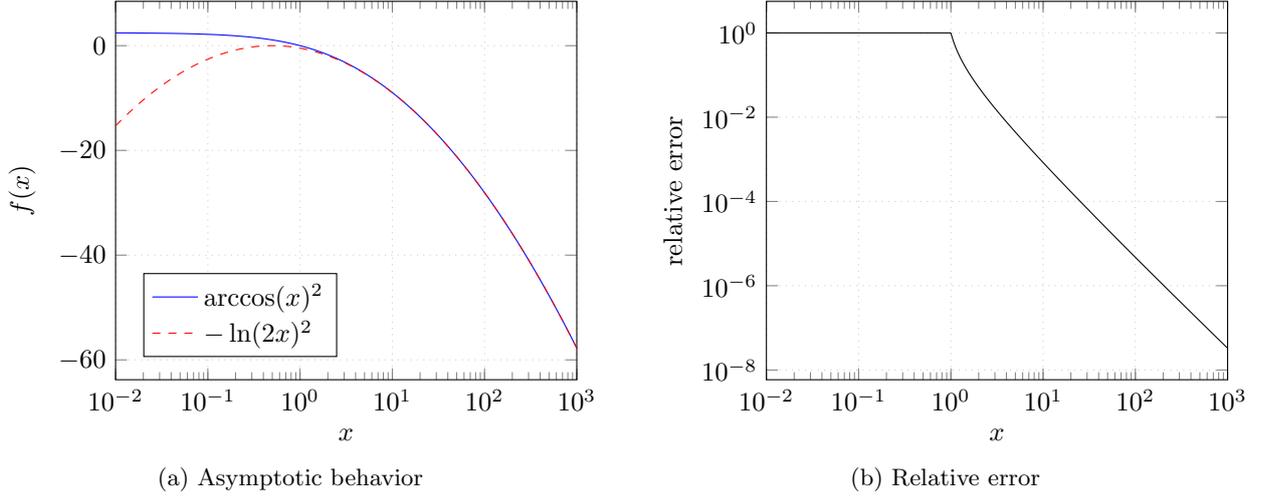
\begin{figure}
    \centering
    \begin{subfigure}[t]{.45\textwidth}
      \centering
      \begin{tikzpicture}[]
        \begin{axis} [
          width=\linewidth,
          xmin=0.01, xmax=1000, xmode=log,
          grid=major,
          grid style={dotted,gray!40},
          legend style={at={(0.06,0.06)},anchor=south west},
          legend cell align=left,
          xlabel=$x$,
          ylabel=$f(x)$
        ]
        \addplot gnuplot [raw gnuplot, id=acos, mark=none, smooth, color=blue]{
          set logscale x;
          \gnuplotdefs
          plot [x=0.01:1000] real(acos(x)**2);
        };
        \addlegendentry{$\arccos(x)^2$};
        \addplot gnuplot [raw gnuplot, id=log, dashed, smooth, mark=none, color=red]{
          set logscale x;
          \gnuplotdefs
          plot [x=0.01:1000] -log(2*x)**2;
        };
        \addlegendentry{$-\ln(2x)^2$};
        \end{axis}
      \end{tikzpicture}
      \caption{Asymptotic behavior}
      \label{fig:arcos_asymptotes}
    \end{subfigure}%
    \hspace{0.05\textwidth}
    \begin{subfigure}[t]{.45\textwidth}
      \centering
      \begin{tikzpicture}[]
          \begin{axis} [
            width=\linewidth,
            xmin=0.01, xmax=1000, xmode=log,
            ymode=log,
            grid=major,
            grid style={dotted,gray!40},
            xlabel=$x$,
            ylabel=relative error,
          ]
          \addplot gnuplot [raw gnuplot, id=log_error, mark=none, color=black]{
            set logscale x;
            \gnuplotdefs
            plot [x=0.01:1000] (real(acos(x)**2) + log(2*x)**2)/(abs(real(acos(x)**2)) + log(2*x)**2);
          };
          \end{axis}
      \end{tikzpicture}
      \caption{Relative error}
      \label{fig:arcos_error}
    \end{subfigure}%
    \caption{Asymptotic behavior (\subref{fig:arcos_asymptotes}) and relative error  (\subref{fig:arcos_error}) for the approximation of $\arccos(x)^2$ by $-\ln(2x)^2$ as in \cref{eq:arccos_approx}}
    \label{fig:arccos_approx}
  \end{figure}

  We have the asymptotics
  \begin{equation}
    \label{eq:arccos_approx}
    \left( \arccos(x) \right)^2 \sim -\left( \ln(2x) \right)^2 \quad (x \rightarrow \infty)
  \end{equation}
  from the fact that $\operatorname{arccos}(x) = \mathrm{i} \operatorname{arccosh}(x) = \mathrm{i} \ln(x + \sqrt{x^2 - 1})$ for $x > 1$ \cite[eq.~4.37.19]{OlverLozierBoisvertClark2010}. \Cref{fig:arccos_approx} shows the accuracy of the approximation.
  Taking the second derivative we have
  \begin{equation*}
    -\frac{1}{2} \frac{\df^2}{\df{x}^2} \left[ \ln(2x) \right]^2 = \frac{1}{x^2} \left(\ln(2 x) - 1 \right).
  \end{equation*}
  Now the interesting solution to $a = \frac{1}{x^2} \left(\ln(2 x) - 1\right)$ (the asymptotics of the inverse of \cref{eq:R_D_arccos} for $x \rightarrow \infty$) is given by
  \begin{equation}
    \label{eq:asymptotics_a_zero}
    x  = \frac{\mathrm{e}}{2} \mathrm{e}^{-\frac{1}{2} W_{-1}\left(-\frac{\mathrm{e}^2 a}{2}\right)} = \sqrt{-\frac{1}{2a} W_{-1}\left(-\frac{\mathrm{e}^2 a}{2}\right)}   \quad (a \rightarrow 0^+).
  \end{equation}
  where $W_{-1}$ denotes the negative branch of the Lambert $W$ function.
  For negative arguments the Lambert $W$ function has actually two solutions/branches.
  We have to choose the branch for which \cref{eq:asymptotics_a_zero} goes to infinity as $a \rightarrow 0^+$, which is $W_{-1}$.
  Due to the approximation of the $\arccos(x)^2$, \cref{eq:asymptotics_a_zero} is valid only for $a \in [2 \mathrm{e}^{-3}, 0]$.
  $W_{-1}$ can be represented by the continued fraction \cite{Veberic2010}
  \begin{equation*}
    W_{-1}(x) = \ln \frac{x}{\ln \frac{x}{\ln \frac{x}{\dots}}}  \quad x \in (-\mathrm{e}^{-1}, 0)
  \end{equation*}
  and for the last $\dots$ one inserts $\ln(-x)$.
  Furthermore we have the limit
  \begin{equation}
    \label{eq:Lambert_W_limit_0}
    \lim_{x\to 0^-} \frac{W_{-1}(x)}{\ln(-x)} = 1.
  \end{equation}
\end{remark}


\begin{remark}[Asymptotics of the inverse for $x \rightarrow 0$]
  The Taylor series at $x=0$ for $\frac{1}{2} \arccos(x)^2$ reads
  \begin{equation*}
    \frac{1}{2} \arccos(x)^2 = \frac{\pi ^2}{8} - \frac{\pi x}{2} + \frac{x^2}{2} - \frac{\pi x^3}{12} + \frac{x^4}{6} - \frac{3 \pi x^5}{80} + \frac{4 x^6}{45} + \mathcal{O}(x^7)
  \end{equation*}
  and its second derivative is given by
  \begin{equation*}
    \frac{1}{2} \frac{\df}{\df x^2} \arccos(x)^2 = 1 - \frac{\pi x}{2} + 2 x^2 - \frac{3 \pi x^3}{4} + \frac{8 x^4}{3} + \mathcal{O}(x^5).
  \end{equation*}
  Then for first order accuracy the solution to $a = 1 - \frac{\pi x}{2}$ is given by
  \begin{equation}
    \label{eq:asymptotics_a_one}
    x = \frac{2}{\pi} (1 - a) \quad (a \rightarrow 1^-)
  \end{equation}
  and for second order accuracy the interesting solution to $a = 1 - \frac{\pi x}{2} + 2 x^2$ is given by
  \begin{equation}
    \label{eq:asymptotics_a_one_2nd}
    x = \frac{1}{8} \left(\pi -\sqrt{32 a+\pi ^2-32}\right).
  \end{equation}
\end{remark}

\section{Application note}

Let $f_n$ denote the density function of the $n$-dimensional angular Gaussian (ACG) distribution \cite{Tyler1987}, i.e.
\begin{equation}
  \label{eq:acg_density}
  f_n(p;\, \Lambda) := \frac{1}{2}\Gamma\left( \frac{n}{2} \right) \frac{1}{\pi^{\frac{n}{2}} |\Lambda|^{\frac{1}{2}}} \left( p^T \Lambda^{-1} p \right)^{-\frac{n}{2}}, \quad p \in \SSd,
\end{equation}
where $\Lambda = B^{-1} \in \R^{n \times n}$ is a symmetric and positive semidefinite distribution parameter and $\SSd$ denotes the surface of the unit sphere in $n$ dimensions. Moments of \cref{eq:acg_density} are given by
\begin{equation}
  \label{eq:moment_tensors}
  \AAn{r} := \int_{\SSd} p^{\otimes r} f_n(p, \Lambda) \, \dx{p}
\end{equation}
where $p^{\otimes r}$ denotes the $r$-fold outer product of a point $p \in \SSd$.
In the following we restrict ourselves to $n=3$ and to the 2nd and 4th order moment
\begin{equation*}
  \begin{aligned}
    \AII &:= \mathbb{A}^{2,3} = \frac{1}{4\pi}  \int_{\SII} (p \otimes p)  \; (B : (p \otimes p))^{-\frac{3}{2}} \dx{p} \quad \text{and}\\
    \AIV &:= \mathbb{A}^{4,3} = \frac{1}{4\pi}  \int_{\SII}  (p \otimes p \otimes p \otimes p) \; (B : (p \otimes p))^{-\frac{3}{2}} \dx{p},
  \end{aligned}
\end{equation*}
which can equivalently be written as elliptic integrals \cite{MontgomerySmith2011}
\begin{equation}
  \label{eq:FEC_b_from_a_0II}
  \AII = \frac{1}{2} \int_0^\infty \!\! \frac{(B + t I)^{-1}}{\sqrt{\det(B + tI)}} \, \dx{t} \quad \text{and}
\end{equation}
\begin{equation}
  \label{eq:FEC_AIV}
  \AIV = \frac{3}{4} \int_0^\infty \frac{t \SYM\!\left( (B + t I)^{-1} \! \otimes (B + t I)^{-1} \right)}{\sqrt{\det(B + t I)}} \, \dx{t},
\end{equation}
where $\SYM$ denotes the symmetrization of a rank-4 tensor, i.e., $\SYM(\mathbb{B})_{ijk\ell}$ is the average of $\mathbb{B}_{mnpq}$ over all 24 permutations $(m,n,p,q)$ of $(i,j,k,\ell)$.
For the simulation of the moments of the orientation distribution of short fibers immersed in a fluid a map from $\AII$ to $\AIV$ is required. This is known as the ``closure problem''.
\Cref{eq:FEC_b_from_a_0II} gives a one-to-one correspondence between $\AII$ and $B$. $B$ can be computed from $A$ using Newton's method for example and subsequently used to compute \cref{eq:FEC_AIV}.
Note that if $\AII$ is diagonal, then also $B$ is diagonal and the eigenbasis of $\AII$ also diagonalizes $B$, thus for diagonalized $\AII = \diag(a)$ and $B=\diag(b)$ \cref{eq:FEC_b_from_a_0II} can be written also component-wise as 
\begin{equation}
  \label{eq:FEC_b_from_a_0II_component_wise}
  a_i = \frac{1}{2} \int_0^\infty \!\! \frac{(t + b_i)^{-1}}{\sqrt{(t + b_1)(t + b_2)(t + b_3)}} \, \dx{t}.
\end{equation}
\Cref{eq:FEC_b_from_a_0II_component_wise} is of the Carlson-type elliptic integral $R_D$ and the inversion problem of \cref{eq:FEC_b_from_a_0II} can be written as
\begin{equation*}
  a_1 = \frac{1}{3} R_D(b_2, b_3, b_1), \quad 
  a_2 = \frac{1}{3} R_D(b_1, b_3, b_2) \quad \text{and} \quad
  a_3 = \frac{1}{3} R_D(b_1, b_2, b_3).
\end{equation*}
We simply write in vectorized form $a = \frac{1}{3} R_D(b)$.
By using properties of the derivative of $R_D$ \cref{eq:FEC_AIV} can further be simplified to
\begin{equation}
  \label{eq:exact_closure}
  \AIV_{iijj} = \frac{1}{2} \begin{cases}
      2 a_{i} -  \frac{a_{k}b_{k} - a_{i}b_{i}}{b_{k}-b_{i}} - \frac{a_{\ell}b_{\ell} - a_{i} b_{i}}{b_{\ell}-b_{i}}  & \text{if $i = j$,}\\
      \frac{a_{i}b_{i} - a_{j}b_{j}}{b_{i}-b_{j}} & \text{if $i \neq j$.}
  \end{cases}
\end{equation}
Due to the normalization $\int_{\SSd} f_n(p, \Lambda) \, \dx{p} = 1$ we have $\tr(\AII) = a_1 + a_2 + a_3 = 1$ as well as $\det(B) = b_1 b_2 b_3 = 1$.

\begin{figure}[t!]
  \begin{center}
    \includegraphics[width=\linewidth]{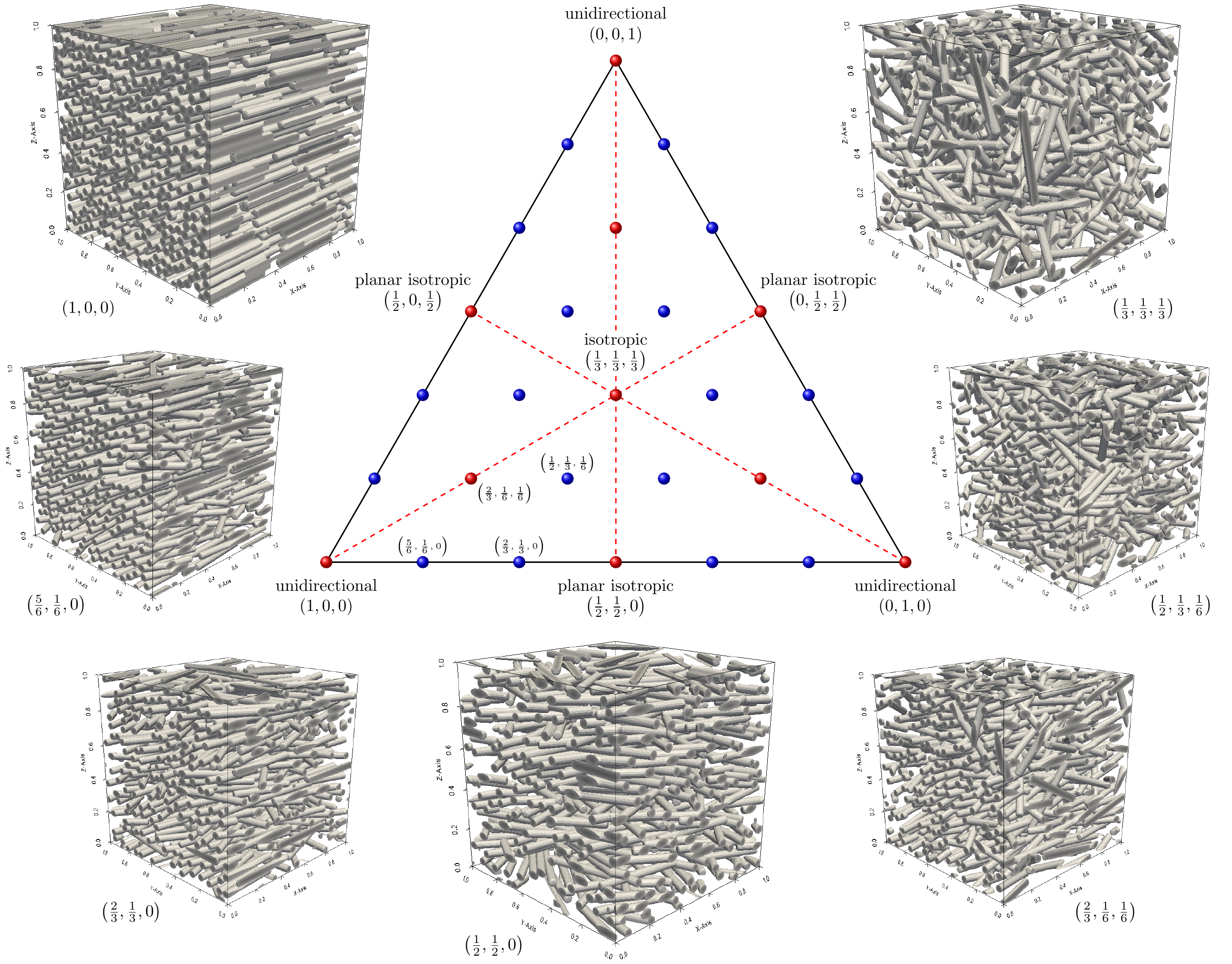}
  \end{center}
  \vspace{-10pt}
  \caption{Visualization of representative moments and their location (eigenvalues) on the unit simplex. The moments on the symmetry lines considered in this article are shown as red dashed lines.}
  \label{fig:fo_simplex}
\end{figure}

\paragraph{Unidirectional orientation states}

In this trivial case the full three-dimensional ACG distribution reduces to a one-dimensional one. Two entries of $a$ are zero and the corresponding elements of $b$ tend to infinity. The remaining entry of $a$, $a_i$ is one and the corresponding element $b_i$ goes to zero.
So clearly $\AII = p \otimes p$ and the closure is given by $\AIV = \AII \otimes \AII = p \otimes p \otimes p \otimes p$ where $p=e_i$ is the $i$-th unit vector.
The unidirectional orientation states are shown in \cref{fig:fo_simplex} as the corner vertices of the unit simplex.

\paragraph{Planar orientation states}

If only one entry of $a$ is zero the ACG distribution reduces to a two-dimensional ACG distribution. The planar orientation states are shown in \cref{fig:fo_simplex} as boundary faces of the unit simplex.
For instance if $a_3 = 0$ then we have according to \cref{eq:acg_density}
\begin{equation*}
    \AII = \mathbb{A}^{2,2} = \frac{1}{2\pi}  \int_{\SII} (p \otimes p)  \; (B : (p \otimes p))^{-1} \dx{p}
\end{equation*}
using the parameterization $p(\varphi) = (\cos(\varphi), \sin(\varphi))$ we get
\begin{equation*}
    a_1 = \frac{1}{2\pi}  \int_0^{2\pi} \frac{\cos^2(\varphi)}{(b_1 \cos^2(\varphi) + b_2 \sin^2(\varphi) )}  \dx{\varphi} = \frac{1 - \sqrt{\frac{b_2}{b_1}}}{b_1-b_2} = \frac{1}{b_1+1}
\end{equation*}
or more generally
\begin{equation*}
    A = (B + I)^{-1} \quad \text{and} \quad B = A^{-1} - I,
\end{equation*}
where we have used $b_1 b_2 = 1$. Inserting $b_i = \frac{1}{a_i} - 1$ into \cref{eq:exact_closure} we arrive at 
\begin{equation}
  \label{eq:exact_closure_planar}
  \AIV_{iijj} = \frac{1}{2} a_i (a_j + \delta_{ij})
\end{equation}
which is valid for the planar orientation states including the unidirectional states.

\paragraph{Non-planar axial-symmetric orientation states}

In the case that exactly two eigenvalues of $\AII$ are equal on can obtain a analytic representation o \cref{eq:FEC_b_from_a_0II_component_wise}. The corresponding orientation states are shown in \cref{fig:fo_simplex} as red dashed lines.
Let $i,j,k \subset \{1,2,3\}$ be distinctive indices, $a \in [0,1]$, $b \in \R$
and $a_i = a$, $a_j = a_k = \frac{1-a}{2}$ and $b_i=b$, $b_j=b_k=\frac{1}{\sqrt{b}}$.
Then we get for the integral \cref{eq:FEC_b_from_a_0II_component_wise}
\begin{equation}
  \label{eq:ax_symm_moments}
  a = \frac{1}{3} R_D\left(\frac{1}{\sqrt{b}}, \frac{1}{\sqrt{b}}, b \right) = \frac{1 - \frac{b^{3/4} \arccos(b^{3/4})}{ \sqrt{1 - b^{3/2}} } }{ 1 - b^{3/2} } = \left. \frac{1}{2} \frac{\df^2}{\df{x}^2} \left[\arccos(x)\right]^2 \right|_{x=b^{3/4}},
\end{equation}
which resembles \cref{eq:R_D_arccos}, noting that $R_D$ is homogeneous of degree $-3/2$. An inversion of \cref{eq:ax_symm_moments} seems to be possible only numerically. However the simple relation to the second derivative of $[\arccos(b)]^2$ is remarkable, which is also the ``message'' of this article.
%
However, using \cref{eq:asymptotics_a_zero} with $x=b^{\frac{3}{4}}$ gives the asymptotic inverse
\begin{equation*}
    b 
    = \left[ \frac{\mathrm{e}}{2} \mathrm{e}^{-\frac{1}{2} W_{-1}\left(-\frac{\mathrm{e}^2 a}{2}\right)}\right]^{\frac{4}{3}} 
    = \left[-\frac{1}{2a} W_{-1}\!\left(-\frac{\mathrm{e}^2 a}{2}\right) \right]^{\frac{2}{3}} \quad (a \rightarrow 0^+).
\end{equation*}
Inserting into \cref{eq:exact_closure} results in
\begin{equation}
  \label{eq:AIV_asymptote_1}
  \AIV_{iiii} 
  = \frac{3 a-1}{2-\frac{\mathrm{e}^2}{2}  \mathrm{e}^{-W_{-1}\left(-\frac{\mathrm{e}^2 a}{2}\right)}} 
  = \frac{3 a-1}{2 + \frac{1}{a} W_{-1}\!\left(-\frac{\mathrm{e}^2 a}{2} \right)} \quad (a \rightarrow 0^+).
\end{equation}
\begin{figure}
  \centering
  \begin{subfigure}[t]{.45\textwidth}
    \centering
    \begin{tikzpicture}[]
      \begin{axis} [
        width=\linewidth,
        xmin=1e-6, xmax=0.1, xmode=log,
        ymode=log,
        yminorticks=false,
        max space between ticks=40,
        grid=major,
        grid style={dotted,gray!40},
        legend style={at={(0.06,0.95)},anchor=north west},
        legend cell align=left,
        xlabel=$a$,
        ylabel=$\AIV_{iiii}$
      ]
        \addplot gnuplot [raw gnuplot, id=asymptote_1_0, mark=none, smooth, color=red]{
          set logscale x;
          \gnuplotdefs
          plot [x=1e-6:0.1] Aiiii_asymptote(x);
        };
        \addlegendentry{\cref{eq:AIV_asymptote_1}};
        \addplot gnuplot [raw gnuplot, id=asymptote_2_0, mark=none, smooth, color=blue]{
          set logscale x;
          \gnuplotdefs
          plot [x=1e-6:0.1] Aiiii_asymptote_2(x);
        };
        \addlegendentry{\cref{eq:AIV_asymptote_2}};
        \addplot gnuplot [raw gnuplot, id=exact_0, mark=none, smooth, dashed, thick, color=black]{
          \gnuplotdefs
          set logscale x;
          set datafile separator ",";
          plot 'b_from_a_0.csv' using ($1):(Aiiii($1,$2));
        };
        \addlegendentry{exact};
      \end{axis}
    \end{tikzpicture}
    \caption{$\AIV_{iiii}$ asymptotics for $a \rightarrow 0$}
    \label{fig:AIV_asymptotes_0}
  \end{subfigure}%
  \hspace{0.05\textwidth}
  \begin{subfigure}[t]{.45\textwidth}
    \centering
    \begin{tikzpicture}[]
        \begin{axis} [
          width=\linewidth,
          xmin=1e-6, xmax=0.1, xmode=log,
          max space between ticks=25,
          ymode=log,
          grid=major,
          grid style={dotted,gray!40},
          legend style={at={(0.94,0.06)},anchor=south east},
          legend cell align=left,
          xlabel=$a$,
          ylabel={relative error},
        ]
        \addplot gnuplot [raw gnuplot, id=asymptote_1_0_error, mark=none, smooth, color=red]{
          \gnuplotdefs
          set logscale x;
          set datafile separator ",";
          plot 'b_from_a_0.csv' using ($1):(Aiiii_asymptote($1)/Aiiii($1,$2) - 1.0);
        };
        \addlegendentry{\cref{eq:AIV_asymptote_1}};
        \addplot gnuplot [raw gnuplot, id=asymptote_2_0_error, mark=none, smooth, color=blue]{
          \gnuplotdefs
          set logscale x;
          set datafile separator ",";
          plot 'b_from_a_0.csv' using ($1):(Aiiii_asymptote_2($1)/Aiiii($1,$2) - 1.0);
        };
        \addlegendentry{\cref{eq:AIV_asymptote_2}};
        \end{axis}
    \end{tikzpicture}
    \caption{$\AIV_{iiii}$ error for $a \rightarrow 0$}
    \label{fig:AIV_error_0}
  \end{subfigure}%
  \caption{Asymptotics (\subref{fig:AIV_asymptotes_0}) and relative error (\subref{fig:AIV_error_0}) of the exact closure \cref{eq:exact_closure} for $a \rightarrow 0$}
  \label{fig:Asymptotes_0}
\end{figure}
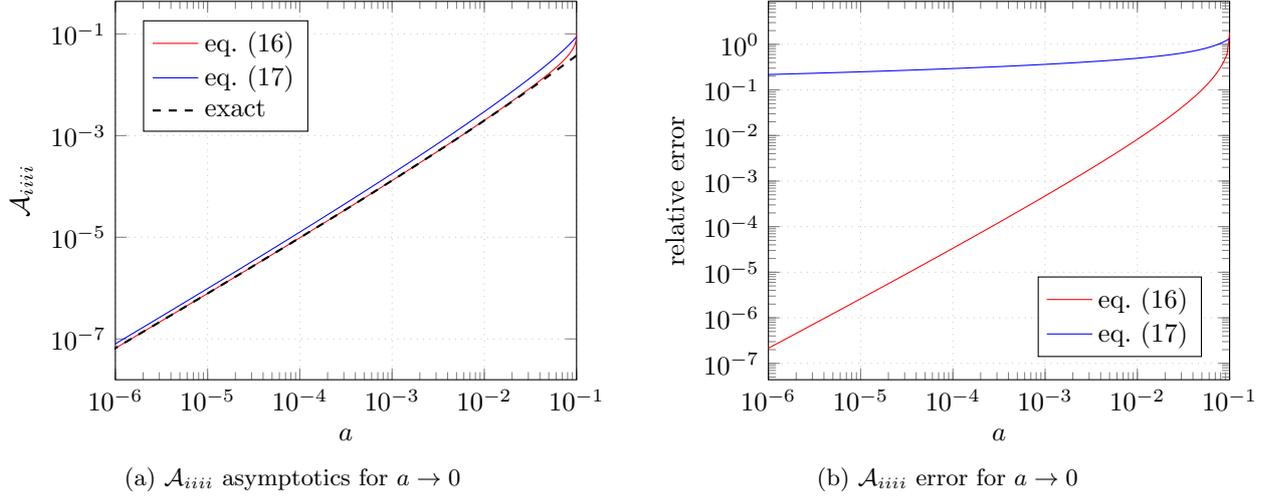%
Note that \cref{eq:AIV_asymptote_1} does not give the correct asymptotic behavior if $W_{-1}(x)$ is further approximated by $\ln(-x)$, i.e.
\begin{equation}
  \label{eq:AIV_asymptote_2}
  \AIV_{iiii} 
  \stackrel{!}{=} \frac{3 a-1}{2 + \frac{1}{a} \ln\!\left(\frac{\mathrm{e}^2 a}{2} \right)} \quad (a \rightarrow 0^+)
\end{equation}
as suggested by \cref{eq:Lambert_W_limit_0}.
As seen in \cref{fig:Asymptotes_0}, \cref{eq:AIV_asymptote_1} shows the correct asymptote when compared to the exact moment computed by numerical inversion, whereas \cref{eq:AIV_asymptote_2} is slightly off and actually does not have the limit $0$ for $a \rightarrow 0$. The error for \cref{eq:AIV_asymptote_2} drops so slowly that it even does not reach zero for $a \rightarrow 0$ and a limit does not exist!
%
%
%
\begin{figure}
  \centering
  \begin{subfigure}[t]{.45\textwidth}
    \centering
    \begin{tikzpicture}[]
        \begin{axis} [
          width=\linewidth,
          xmin=0.0001, xmax=0.30, xmode=log,
          grid=major,
          grid style={dotted,gray!40},
          legend style={at={(0.06,0.06)},anchor=south west},
          legend cell align=left,
          xlabel=$1-a$,
          ylabel=$\AIV_{iiii}$
        ]
        \addplot gnuplot [raw gnuplot, id=asymptote_4_1, mark=none, smooth, color=red]{
          set logscale x;
          \gnuplotdefs
          plot [x=0.0001:0.30] Aiiii_asymptote_4(1-x);
        };
        \addlegendentry{\cref{eq:AIV_asymptote_4}};
        \addplot gnuplot [raw gnuplot, id=asymptote_5_1, mark=none, smooth, color=blue]{
          set logscale x;
          \gnuplotdefs
          plot [x=0.0001:0.30] Aiiii_asymptote_5(1-x);
        };
        \addlegendentry{\cref{eq:AIV_asymptote_5}};
        \addplot gnuplot [raw gnuplot, id=exact_1, mark=none, smooth, dashed, thick, color=black]{
          \gnuplotdefs
          set logscale x;
          set datafile separator ",";
          plot 'b_from_a_1.csv' using (1-$1):(Aiiii($1,$2));
        };
        \addlegendentry{exact};
      \end{axis}
    \end{tikzpicture}
    \caption{$\AIV_{iiii}$ asymptotics for $a \rightarrow 1$}
    \label{fig:AIV_asymptotes_1}
  \end{subfigure}%
  \hspace{0.05\textwidth}
  \begin{subfigure}[t]{.45\textwidth}
    \centering
    \begin{tikzpicture}[]
      \begin{axis} [
        width=\linewidth,
        xmin=0.0001, xmax=0.3, xmode=log,
        ymode=log,
        grid=major,
        grid style={dotted,gray!40},
        legend style={at={(0.94,0.06)},anchor=south east},
        legend cell align=left,
        xlabel=$1-a$,
        ylabel={relative error}
      ]
        \addplot gnuplot [raw gnuplot, id=asymptote_4_1_error, mark=none, smooth, color=red]{
          \gnuplotdefs
          set datafile separator ",";
          plot 'b_from_a_1.csv' using (1-$1):(Aiiii($1,$2)/Aiiii_asymptote_4($1) - 1.0);
        };
        \addlegendentry{\cref{eq:AIV_asymptote_4}};
        \addplot gnuplot [raw gnuplot, id=asymptote_5_1_error, mark=none, smooth, color=blue]{
          \gnuplotdefs
          set datafile separator ",";
          plot 'b_from_a_1.csv' using (1-$1):(1 - Aiiii($1,$2)/Aiiii_asymptote_5($1));
        };
        \addlegendentry{\cref{eq:AIV_asymptote_5}};
      \end{axis}
    \end{tikzpicture}
    \caption{$\AIV_{iiii}$ error for $a \rightarrow 1$}
    \label{fig:AIV_error_1}
  \end{subfigure}%
  \caption{Asymptotics (\subref{fig:AIV_asymptotes_0}) and relative error  (\subref{fig:AIV_error_0}) of the exact closure \cref{eq:exact_closure} for $a \rightarrow 1$}
  \label{fig:Asymptotes_1}
\end{figure}
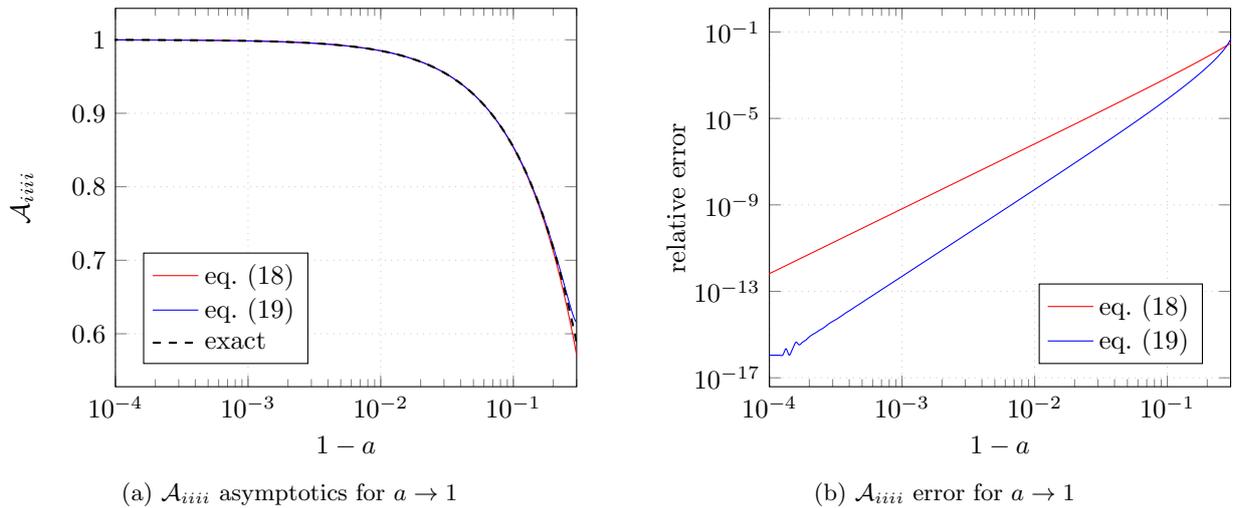%
For the asymptotes $a \rightarrow 1$ we proceed similarly: \cref{eq:asymptotics_a_one} with $x=b^{\frac{3}{4}}$ gives
\begin{equation*}
    b = \left[ \frac{2}{\pi} (1 - a) \right]^{\frac{4}{3}} \quad (a \rightarrow 1^-)
\end{equation*}
and inserting this into \cref{eq:exact_closure} results in
\begin{equation}
  \label{eq:AIV_asymptote_4}
  \AIV_{iiii} = \frac{3 a - 1}{2 - \frac{8}{\pi ^2} (a-1)^2} \quad (a \rightarrow 1^-)
\end{equation}
Similarly we get for second order accuracy using \cref{eq:asymptotics_a_one_2nd}
\begin{equation}
  \label{eq:AIV_asymptote_5}
  \AIV_{iiii} = \frac{3 a-1}{2-\frac{1}{32}\left(\pi - \sqrt{32(a-1) + \pi^2}\right)^2}.
\end{equation}
The behavior and relative error is shown for both equations in \cref{fig:Asymptotes_1}.

\section{Conclusions and outlook}
\label{sec:outlook}

Using the relation \cref{eq:R_D_arccos} which is, to the authors' knowledge, not explicitly documented elsewhere, we were able to 
derive different asymtotes for the exact closure.
An explicit (series) expression of $\AIV$ in terms of $\AII$ for the non-planar axial-symmetric orientation states as well as other full three dimensional orientation states is still unknown. 
If such an expression exists it should include the planar and unidirectional states as well, i.e.\ it should coincide with \cref{eq:exact_closure_planar} for the planar limit.
As we can see from the different asymtotes \cref{eq:AIV_asymptote_1,eq:AIV_asymptote_2,eq:AIV_asymptote_4,eq:AIV_asymptote_5}, the general 3d case assumes a complex behavior between hyperbolic and rational functions.
In the future the authors will further investigate the asymptotic behavior outside of the symmetry lines.
Further we will look into the numerical approximation of the full tensor $\AIV$ in terms of $\AII$ including a tensor representation of \cref{eq:exact_closure_planar}, without the requirement of an eigen-decomposition of the second moment.




\section*{Acknowledgments}

This work was performed within the Federal Cluster of Excellence EXC 1075 ``MERGE Technologies for Multifunctional Lightweight Structures'' and supported by the German Research Foundation (DFG). Financial support is gratefully acknowledged.
 
\medskip


\printbibliography

\end{document}